\documentclass[12pt]{amsart}
\usepackage{amssymb}
\usepackage{mathptmx}

\usepackage[latin1]{inputenc}
\usepackage[all]{xy}

\newcommand{\mm}{\mathfrak m}
\newcommand{\N}{\mathbb{N}}
\newcommand{\PP}{\mathbb{P}}

\newcommand{\Z}{\mathbb{Z}}
\newcommand{\bo}{{\bf 0}}

\DeclareMathOperator{\chara}{char}

\DeclareMathOperator{\depth}{depth}

\DeclareMathOperator{\Ext}{Ext}

\DeclareMathOperator{\Hom}{Hom}

\DeclareMathOperator{\reg}{reg}

\DeclareMathOperator{\signum}{sign}

\DeclareMathOperator{\Sym}{Sym} 
\DeclareMathOperator{\Tor}{Tor} \DeclareMathOperator{\slope }{slope}
\DeclareMathOperator{\Br}{Rate} \DeclareMathOperator{\ind}{index}

\DeclareMathOperator{\dirsum}{\oplus}
 
\DeclareMathOperator{\tensor}{\otimes}

\DeclareMathOperator{\pnt}{\raise 0.5mm \hbox{\large\bf.}}

\newtheorem{thm}{\bf Theorem}[section]
\newtheorem{lem}[thm]{\bf Lemma}
\newtheorem{cor}[thm]{\bf Corollary}
\newtheorem{prop}[thm]{\bf Proposition}

\theoremstyle{definition}

\newtheorem{rem}[thm]{\bf Remark}
\newtheorem{ex}[thm]{\bf Example}

\let\signum=\relax

\title{Koszul homology and syzygies of Veronese subalgebras}

\author{Winfried Bruns}
\address{Universit\"at Osnabr\"uck, Institut f\"ur Mathematik, 49069 Osnabr\"uck, Germany}
\email{wbruns@uos.de}

\author{Aldo Conca}
\address{Dipartimento di Matematica, Universit\'{a} di Genova, Via Dodecaneso 35, 16146 Genova,
Italy} \email{conca@dima.unige.it}

\author{Tim R\"omer}
\address{Universit\"at Osnabr\"uck, Institut f\"ur Mathematik, 49069 Osnabr\"uck, Germany}
\email{troemer@uos.de}

\dedicatory{To J\"urgen Herzog, friend and teacher}

\begin{document}

\begin{abstract}
A graded $K$-algebra $R$ has property $N_p$ if it is generated in
degree $1$, has relations in degree $2$ and the syzygies of order
$\leq p$ on the relations are linear. The Green-Lazarsfeld index of
$R$ is the largest $p$ such that it satisfies the property $N_p$.
Our main results assert that (under a mild assumption on the base field)
the $c$-th Veronese subring of a polynomial ring has
Green-Lazarsfeld index $\geq c+1$. The same conclusion also holds
for an arbitrary standard graded algebra, provided $c\gg 0$.
\end{abstract}

%\subjclass{}
%\keywords{}

\maketitle

%------------------------------------------------------------------------
%
%
%
%------------------------------------------------------------------------
\section{Introduction}
\label{sec:intro}

A classical problem in algebraic geometry and commutative algebra is
the study of the equations defining projective varieties and of
their syzygies. Green and Lazarsfeld \cite{GRLA, GRLA88} introduced
the property $N_p$ for a graded ring as an indication of the
presence of simple syzygies. Let us recall the definition. A
finitely generated $\N$-graded $K$-algebra $R=\dirsum_{i} R_i$ over
a field $K$ satisfies property $N_0$ if $R$ is generated (as a
$K$-algebra) in degree $1$. Then $R$ can be presented as a quotient
of a standard graded polynomial ring $S$ and one says that $R$
satisfies \emph{property $N_p$} for some $p>0$ if
$\beta_{i,j}^S(R)=0$ for $j>i+1$ and $1\leq i \leq p$. Here
$\beta_{i,j}^S(R)$ denote the graded Betti numbers of $R$ over $S$.
For example, property $N_1$ means that $R$ is defined by quadrics,
$N_3$ means that $R$ is defined by quadrics and that the first and
second syzygies of the quadrics are linear. We define the
\emph{Green-Lazarsfeld index} of $R$, denoted by $\ind(R)$, to be
the largest $p$ such that $R$ has $N_p$, with $\ind(R)=\infty$ if
$R$ satisfies $N_p$ for every $p$. It is, in general, very difficult
to determine the precise value of the Green-Lazarsfeld index.
Important conjectures, such as Green's conjecture on the syzygies of
canonical curves \cite[Chap.9]{E}, predict the value of the
Green-Lazarsfeld index for certain families of varieties.

The goal of this paper is to study the Green-Lazarsfeld index of the
Veronese embeddings $v_c:\PP^{n-1}\to \PP^N$ of degree $c$ of
projective spaces and, more generally, of the Veronese embeddings of
arbitrary varieties. Let $S$ denote the polynomial ring in $n$
variables over the field $K$. The coordinate ring of the image of
$v_c$ is the Veronese subring $S^{(c)}=\bigoplus_{i\in \N} S_{ic}$
of $S$. If $n\leq 2$ or $c\leq 2$ then $S^{(c)}$ is a determinantal
ring whose resolution is well understood and the Green-Lazarsfeld
index can be easily deduced. If $n=2$ then $S^{(c)}$ is resolved by
the Eagon-Northcott complex and so $\ind(S^{(c)})=\infty$. The
resolution of $S^{(2)}$ in characteristic $0$ is described by
Jozefiak, Pragacz and Weyman in \cite{JPW}; it follows that
$\ind(S^{(2)})=5$ if $n>3$ and $\ind(S^{(2)})=\infty$ if $n\leq 3$.
For $n\leq 6$ Andersen \cite{And} proved that  $\ind(S^{(2)})$ is
independent on $\chara K$, but for $n>6$ and $\chara K=5$ she proved
that $\ind(S^{(2)})=4$.

For $n>2$ and $c>2$ it is known that
\begin{equation}
\label{OPbounds} c\leq \ind(S^{(c)})\leq 3c-3.
\end{equation}
The lower bound is due to Green \cite{GR2} (and holds for any $c$
and $n$). Ottaviani and Paoletti \cite{OTPA} established the upper
bound in characteristic $0$. They also showed that
$\ind(S^{(c)})=3c-3$ for $n=3$ and conjectured that $
\ind(S^{(c)})=3c-3 \text{ holds true for arbitrary } n\geq 3$; see also
Weyman \cite[Proposition 7.2.8]{WEY}. For
$n=4$ and $c=3$ it is indeed the case \cite[Lemma 3.3]{OTPA}. In their
interesting paper \cite{EGHP} Eisenbud, Green, Hulek and Popescu
reproved some of these statements using different methods. Rubei
\cite{RU04} proved that $\ind(S^{(3)})\geq 4$ if $\chara K=0$. Our
main results are the following:

 \begin{itemize}
 \item[(i)] $c+1\leq \ind(S^{(c)})$ if $\chara K=0$ or $>c+1$; see
 Corollary \ref{cor:ncpolynomial}.

\item[(ii)] If $R$ is a quotient of $S$ then $\ind(R^{(c)})\geq \ind(S^{(c)})$
for every $c\geq \slope _S(R)$; see Theorem \ref{thm:rateresult}. In
particular, if $R$ is Koszul then $\ind(R^{(c)})\geq \ind(S^{(c)})$
for every $c\geq 2$,
\end{itemize}

Furthermore we give characteristic free proofs of the bounds
(\ref{OPbounds}) and of the equality for $n=3$; see Theorem
\ref{thm:ottpa}. Our approach is based on the study of the Koszul
complex associated to the $c$-th power of the maximal ideal. Let $R$
be a standard graded $K$-algebra with maximal homogeneous ideal
$\mm$. Let $K(\mm^c,R)$ denote the Koszul complex associated to
$\mm^c$, $Z_t(\mm^c,R)$ the module of cycles of homological degree
$t$ and $H_t(\mm^c,R)$ the corresponding homology module. In Section
\ref{sec:generalbounds} we study the homological invariants of
$Z_t(\mm^c,R)$. Among other facts, we prove, in a surprisingly
simple way, a generalization of Green's theorem \cite[Thm. 2.2]{GR2}
to arbitrary standard graded algebras; see Corollary
\ref{greengoes}. If $R$ is a polynomial ring (or just a Koszul
ring), then it follows that the regularity of $Z_t(\mm^c,R)$ is
$\leq t(c+1)$; see Proposition \ref{boundZ}.

In Section \ref{sec:cycles} we investigate more closely the modules
$Z_t(\mm^c,S)$ in the case of a polynomial ring $S$. Lemma
\ref{lem:newcycles} describes certain cycles which then are
used to prove a vanishing statement in Theorem
\ref{thm:morevanishing}. In Section \ref{sec:polynomial} we improve
the lower bound (\ref{OPbounds}) by one, see Corollary
\ref{cor:ncpolynomial}. Proposition \ref{prop:duality} states
a duality of Avramov-Golod type, which is the algebraic counterpart
of Serre duality. The duality is then used to establish Ottaviani
and Paoletti's upper bound $\ind(S^{(c)})\leq 3c-3$ in arbitrary
characteristic (Theorem \ref{thm:ottpa}). We also show that for
$n=3$ one gets $\ind(S^{(c)})= 3c-3$ independently of the
characteristic; see Theorem \ref{thm:ottpa}.

In Section \ref{sec:standard} we take $R$ to be a quotient of a
Koszul algebra $D$ and prove that for every $c\geq \slope _D(R)$ we
have $\ind(R^{(c)})\geq \ind(D^{(c)})$; see Theorem
\ref{thm:rateresult}. Here $\slope _D(R)=\sup\{ t_i^D(R)/i : i\geq 1
\}$ where $t_i^D(R)$ is the largest degree of an $i$-th syzygy of
$R$ over $D$. In particular, $\slope _D(R)=2$ if $R$ is Koszul
(Avramov, Conca and Iyengar \cite{ACI}) and, when $D=S$ is a
polynomial ring, $\slope _S(R)\leq a$ if the defining ideal of $R$
has a Gr\"obner basis of elements of degree $\leq a$.  Similar
results have been obtained by Park \cite{PP} under some restrictive
assumptions.   In the last section we discuss multigraded variants
of the results presented.

\section{General bounds }
\label{sec:generalbounds}

In this section we consider a standard graded $K$-algebra $R$ with
maximal homogeneous ideal $\mm$, which is a quotient of a polynomial
ring $S$, say $R=S/I$ where $I$ is homogeneous (and may contain
elements of degree $1$). For a finitely generated graded $R$-module
$M$ let $\beta_{i,j}^R(M)=\dim_K \Tor_i^R(M,K)_j$ be the
\emph{graded Betti numbers} of $M$ over $R$. We define
the number
$$
t_i^R(M)=\max\{j \in \Z : \beta^R_{i,j}(M)\neq 0\},
$$
if $\Tor_i^R(M,K)\neq 0$ and $t_i^R(M)=-\infty$ otherwise. The
\emph{(relative) regularity of $M$ over $R$} is given by
$$
\reg_R(M)=\sup \{ t^R_i(M)-i : i\in \N \}
$$
and the \emph{Castelnuovo-Mumford regularity} of $M$ is
$$
\reg(M)=\reg_S(M)=\sup \{ t^S_i(M)-i : i\in \N \};
$$
it has also the cohomological interpretation
$$
\reg(M)=\max\{ j : H_\mm^i(M)_{j-i}\neq 0 \mbox{ for some } i\in \N\}
$$
where $H_\mm^i(M)$ denotes the $i$-th local cohomology module of
$M$. One defines the \emph{ slope } of $M$ over $R$ by
$$
\slope _R(M)=\sup \{ \frac{ t^R_i(M)-t^R_0(M)}{i} : i\in \N, i>0\},
$$
and the \emph{Backelin rate} of $R$ by
$$
\Br(R)=\slope _R(\mm)=\sup \{ \frac{t^R_{i}(K)-1}{i-1} : i\in \N,
i>1\}.
$$
The Backelin rate measures the deviation from being Koszul: in
general, $\Br(R)\geq 1$, and $R$ is Koszul if and only if $\Br(R)=1$.
Finally, the \emph{Green-Lazarsfeld index} of $R$ is given by
$$
\ind(R)=\sup\{ p\in \N : t_i^S(R)\leq i+1 \mbox{ for every } i\leq p\}.
$$
It is the largest non-negative integer $p$ such that $R$ satisfies
the property $N_p$. Note that we have $\ind(R)=\infty$ if and only
if $\reg(R)\leq 1$, that is, the defining ideal of $R$ has a
$2$-linear resolution. On the other hand, $\ind(R)\geq 1$ if and
only if $R$ is defined by quadrics. In general, $\reg(M)$ and
$\slope _R(M)$ are finite (see \cite{ACI}) while $\reg_R(M)$ can be
infinite. However, $\reg_R(M)$ is finite if $R$ is Koszul, see
Avramov and Eisenbud \cite{AE}.

\begin{rem}
The invariants $\reg(M)$ and $\ind(R)$ are defined in terms of a
presentation of $R$ as a quotient of a polynomial ring but do not
depend on it. The assertion is a consequence of the following
formula which is obtained, for example, from the graded analogue of
\cite[Theorem 2.2.3]{AV96}: if $x\in R_1$ and $xM=0$, then
$$
\beta_{i,j}^R(M) = \beta_{i,j}^{R/(x)}(M) +
\beta_{i-1,j-1}^{R/(x)}(M) .
$$
\end{rem}

We record basic properties of these invariants. The modules are
graded and finitely generated and the homomorphisms are of degree
$0$.

\begin{lem}
\label{lem:tiineq} Let $R$ be a standard graded $K$-algebra, $N$ and
$M_j$, $j\in \N$, be $R$-modules and $i\in \N$.
\begin{enumerate}
\item
Let
$$
0\to M_1\to M_2\to M_3\to 0
$$
be an exact sequence. Then
\begin{align*}
t^R_i(M_1)&\leq \max\{ t^R_i(M_2) , t^R_{i+1}(M_3)\},\\
t^R_i(M_2)&\leq \max\{ t^R_i(M_1) , t^R_i(M_3)\},\\
t^R_i(M_3)&\leq \max\{ t^R_i(M_2) , t^R_{i-1}(M_1)\}.
\end{align*}
\item
Let
$$
\cdots \to M_{k+1}\to M_{k}\to M_{k-1}\to \cdots \to M_1 \to M_0\to
N\to 0
$$
be an exact complex. Then
\begin{align*}
t^R_i(N)&\leq \max\{ t^R_{i-j}(M_j) : j=0,\dots ,i\}
 \intertext{and}
\reg_R(N)&\leq \sup\{ \reg_R(M_j)-j : j\geq 0\}.
\end{align*}

\item If $N$ vanishes in degree $>a$ then $t^R_i(N)\leq t^R_i(K)+a$.

\item Let $J$ be a homogeneous ideal of $R$. If
$\reg_R(R/J)=0$, then
$$
\ind(R/J)\geq \ind(R).
$$
\end{enumerate}
\end{lem}

\begin{proof}
To prove (a) one just considers the long exact homology sequence for
$\Tor^R(\cdot,K)$. For (b) one uses induction on $i$ and
applies (a). Part (c) is proved by induction on $a-\min\{ j :
N_j\neq 0\}$. For (d) one applies (c) to the minimal free resolution
of $R/J$ as an $R$-module. For every $i$ one gets $t_i^S(R/J)\leq
\max\{ t^S_{i-j}(R(-j)) : j=0,\dots i\}$. But we have
$t^S_{i-j}(R(-j))=t^S_{i-j}(R)+j$. If $i\leq \ind(R)$ then
$t^S_{i-j}(R)\leq i-j+1$. It follows that $t_i^S(R/J)\leq i+1$ for
every $i\leq \ind(R)$. Hence $\ind(R/J)\geq \ind(R)$.
\end{proof}

Let $M$ be an $R$-module and let $K(\mm^c,M)$ be the graded Koszul
complex associated to the $c$-th power of the maximal ideal of $R$.
Let $Z_i(\mm^c,M),$ $B_i(\mm^c,M),$ $H_i(\mm^c,M)$ denote the
$i$-th cycles, boundaries and homology of $K(\mm^c,M)$,
respectively. We have:

\begin{lem}\label{veryeasy}
Set $Z_i=Z_i(\mm^c,M)$. For every $a\geq 0$ and $i\geq 0$ we have:
\begin{align*}
t^R_a ( Z_{i+1})\leq\max\bigl\{&t^R_a(M)+(i+1)c, \\
       &t^R_{a+1}(Z_i), \\
       &t^R_{0}(Z_i)+c+(a+1)\Br(R)\bigr\}.
\end{align*}
\end{lem}

\begin{proof} Set $B_i=B_i(\mm^c, M)$ and $H_i=H_i(\mm^c,M)$.
Recall that $\mm^cH_i=0$ and hence $H_i$ vanishes in degrees $>
t_0(Z_i)+c-1$. It follows from Lemma \ref{lem:tiineq}(c) that
$$t_a^R(H_i)\leq t^R_0(Z_i)+c-1+t_a^R(K).$$
The short exact sequences
$$0
\to B_i \to Z_i \to H_i \to 0
$$
and
$$
0\to Z_{i+1} \to K_{i+1} \to B_i \to 0
$$
now imply that
\begin{align*}
t^R_a(Z_{i+1})&\leq
\max\{ t^R_a(M)+(i+1)c, \ t^R_{a+1}(B_i)\}\\
&\leq \max\{ t^R_a(M)+(i+1)c, \ t^R_{a+1}(Z_i), \ t^R_{a+2}(H_i) \} \\
&\leq \max\{ t^R_a(M)+(i+1)c, \ t^R_{a+1}(Z_i), \
t^R_0(Z_i)+c-1+t^R_{a+2}(K) \}.
\end{align*}
Since, by the very definition, $t^R_{a+2}(K)\leq 1+(a+1)\Br(R)$ the
desired result follows.
\end{proof}

Lemma \ref{veryeasy} allows us to bound $t^R_a(Z_i)$
inductively in terms of the various $t^R_{j}(M)$ and of
$\Br(R)$:

\begin{prop}\label{boundZ}
Set $Z_i=Z_i(\mm^c,M)$.
\begin{enumerate}
\item
Assume $c\geq \slope _R(M)$. Then for all $a,i\in \N$ we have
$$
t^R_a(Z_i)\leq t^R_0(M)+ic+\max\{a \slope _R(M), (a+i)\Br(R)\}.
$$
In particular, $\slope _R (Z_i)\leq\max\{\slope _R M , (1+i) \Br(R) \}$.

\item Assume $R$ is Koszul, i.e., $\Br(R)=1$. Then for all $a,i\in \N$ we
have
$$
t^R_a(Z_i)\leq a+i(c+1)+\reg_R(M).
$$
In particular, $\reg_R (Z_i)\leq i(c+1)+\reg_R(M)$.
\end{enumerate}
\end{prop}

\begin{proof}
To show (a) one uses that $t_a^R(M)\leq t_0^R(M)+a\slope _R(M)$ in
combination with Lemma \ref{veryeasy} and induction on $i$. For (b)
one observes that $t_a^R(M)\leq a+\reg_R(M)$ in combination with
Lemma \ref{veryeasy} and induction on $i$.
\end{proof}

In particular, we have the following corollary that generalizes
Green's theorem \cite[Theorem 2.2]{GR2} to arbitrary standard graded
$K$-algebras.

\begin{cor}\label{greengoes}
 Set $Z_i=Z_i(\mm^c,R)$. Then:
\begin{enumerate}
\item $t_0^R(Z_i)\leq ic+\min\{i \Br(R), i+\reg(R)\}$.
\item $H_i(\mm^c,R)_{ic+j}=0$ for every $j\geq \min\{i \Br(R), i+\reg(R)\}+c$.
\end{enumerate}
\end{cor}

\begin{proof}
To prove (a) one notes that setting $M=R$ and $a=0$ in Proposition \ref{boundZ}
(a) one has $t_0^R(Z_i)\leq ic+i\Br(R)$. Then one considers $R$ as
an $S$-module and sets $M=R$ and $a=0$ in Proposition \ref{boundZ} (b). One has
$t_0^S(Z_i)\leq i(c+1)+\reg(R)$. Since $t_0^S(Z_i)=t_0^R(Z_i)$ we
are done. To prove (b) one uses (a) and the fact that
$\mm^cH_i(\mm^c,R)=0$.
\end{proof}

\section{Koszul cycles} \label{sec:cycles}

In this section we concentrate our attention on the Koszul complex
$K(\mm^c)=K(\mm^c,S)$ where $S=K[X_1,\ldots,X_n]$ is a standard
graded polynomial ring over a field $K$ and
$\mm=(X_1,\dots,X_n)$ is its maximal homogeneous ideal. The Koszul
complex $K(\mm^c)$ is indeed an $S$-algebra, namely the exterior
algebra $S\tensor_K \bigwedge^{\pnt} S_c\cong\bigwedge^{\pnt}
F$ where $F$ is the free $S$-module of rank equal to $\dim_K
S_c=\binom{n-1+c}{n-1}$. The differential of $K(\mm^c)$ is denoted
by $\partial$; it is an antiderivation of degree $-1$. We consider
the cycles $Z_t(\mm^c,S)$, simply denoted by $Z_t(\mm^c)$, of the
Koszul complex $K(\mm^c)$, and the $S$-subalgebra
$Z(\mm^c)=\bigoplus_t Z_t(\mm^c)$ of $K(\mm^c)$.

For $f_1,\dots,f_t\in S_c$ and $g\in S$ we set
$$
g[f_1,\dots,f_t]=g\tensor f_1\wedge \dots \wedge f_t\in
K_t(\mm^c).
$$
The elements $[u_1,\dots,u_t]$ for distinct monomials
$u_1,u_2,\dots,u_t$ of degree $c$ (ordered in some way) form a basis
of $K_t(\mm^c)$ as an $S$-module. We call them \emph{monomial free
generators} of $K_t(\mm^c)$. The elements $v[u_1,\dots,u_t]$, where
$u_1,u_2,\dots,u_t$ are distinct monomials of degree $c$ and $v$ is
a monomial of arbitrary degree, form a basis of the $K$-vector space
$K_t(\mm^c)$. They are called \emph{monomials} of $K_t(\mm^c)$.
Evidently $K(\mm^c)$ is a $\Z$-graded complex, but it is also
$\Z^n$-graded with the following assignment of degrees: $\deg
v[u_1,\dots,u_t]=\alpha$ where $vu_1\cdots u_t=X^\alpha$.

Every element $z\in K_t(\mm^c)$ can be written uniquely as a linear
combination
$$
z=\sum f_i [u_{i1},\dots,u_{it}]
$$
of monomial free generators $[u_{i1},\dots,u_{it}]$ with
coefficients $f_i\in S$. We call $f_i$ the \emph{coefficient} of
$[u_{i1},\dots,u_{it}]$ in $z$. Note that $z$ is $\Z$-homogenous of
degree $ct+j$ if every $f_i$ is homogeneous of degree $j$. In this
case $z$ has \emph{coefficients of degree $j$}. Note also that $z$
is homogeneous of degree $\alpha\in \Z^n$ in the $\Z^n$-grading if
for every $i$ one has $f_i=\lambda_i v_i$ such that $\lambda_i\in K$
and $v_i$ is a monomial with $v_iu_{i1}\cdots u_{it}=X^{\alpha}$.
Given $z\in K(\mm^c)$ and a monomial $v[u_1,\dots,u_t]$ we say that
$v[u_1,\dots,u_t]$ \emph{appears in} $z$ if it has a non-zero
coefficient in the representation of $z$ as $K$-linear combination
of monomials of $K(\mm^c)$. An immediate consequence of Proposition
\ref{boundZ} is:

\begin{lem}\label{regZ}
We have $\reg(Z_t(\mm^c)) \leq t(c+1)$. In particular, $Z_t(\mm^c)$
is generated by elements of degree $\leq t(c+1)$.
\end{lem}

\begin{rem}\label{genZ1}
It is easy to see and well known that $Z_1(\mm^c)$ is generated by
the elements $X_i[X_jb]-X_j[X_ib]$ where $b$ is a monomial of degree
$c-1$.

\end{rem}
We write $Z_1(\mm^c)^t$ for $\bigwedge^t Z_1(\mm^c)\subset
Z_t(\mm^c)$, and similarly for other products.

\begin{thm}
\label{thm:generators} For every $t$ the module
$Z_t(\mm^c)/Z_1(\mm^c)^t$ is generated in degree $<t(c+1)$.
\end{thm}

\begin{proof}
The assertion is proved by induction on $t$. For $t=1$ there is
nothing to do. By induction it is enough to verify that
$Z_t(\mm^c)/Z_1(\mm^c)Z_{t-1}(\mm^c)$ is generated in degree
$<t(c+1)$. Since $Z_t(\mm^c)$ is $\Z^n$-graded and generated in
degree $\leq t(c+1)$, it suffices to show that every $\Z^n$-graded
element $f\in Z_t(\mm^c)$ of total degree $t(c+1)$ can be written
modulo $Z_1(\mm^c)Z_{t-1}(\mm^c)$ as a multiple of an element in
$Z_t(\mm^c)$ of total degree $<t(c+1)$. Let $\alpha\in \Z^n$ be the
$\Z^n$-degree of $f$. Permuting the coordinates if necessary, we may
assume $\alpha_n>0$.

Let $u\in S$ be a monomial of degree $c$ with $X_n\mid u$. We write
$f=a+b[u]$ with $a \in K_{t}(\mm^c)$ and $b\in K_{t-1}(\mm^c)$ such
that $a,b$ involve only free generators $[u_1,\dots, u_s]$
($s=t,t-1$) with $u_i \neq u$ for all $i$. Since
$$
0=\partial(f)=\partial(a)+\partial(b)[u] \pm b u
$$
it follows that $\partial(b)=0$, i.e., $b\in Z_{t-1}(\mm^c)$. Note
that $b$ has coefficients of degree $t$. Since $Z_{t-1}(\mm^c)$ is
generated by elements with coefficients of degree $\leq t-1$ we may
write
\begin{equation}
\label{eq:helper} b=\sum_{j=1}^s \lambda_j v_j z_j
\end{equation}
where $\lambda_j\in K$, $z_j\in Z_{t-1}(\mm^c)$ and the $v_j$ are
monomials of positive degree.

Let $\lambda_j v_j z_j$ be a summand in (\ref{eq:helper}). If $X_n$
does not divide $v_j$, then choose $i<n$ such that $X_i\mid v_j$. We
set $z' = X_i [u]-X_n [u'] \in Z_1(\mm^c)$ where $u'=u X_i/X_n$, and
subtract from $f$ the element
$$
\lambda_j \frac{v_j}{X_i} z_j z' \in Z_{t-1}(\mm^c)Z_1(\mm^c).
$$
Repeating this procedure for each $\lambda_j v_j z_j$ in
(\ref{eq:helper}) such that $X_n$ does not divide $v_j$ we obtain a
cycle $f_1\in Z_t(\mm^c)$ of degree $\alpha$ such that
\begin{itemize}
\item[(i)] $f=f_1 \mod Z_1(\mm^c)Z_{t-1}(\mm^c)$;
\item[(ii)] if a monomial $v [u_1,\dots,u_t]$ appears in $f_1$
and $u\in \{ u_1,\dots,u_t\}$, then $X_n\mid v$.
\end{itemize}
We repeat the described procedure for each monomial $u$ of degree
$c$ with $X_n|u$. We end up with an element $f_2\in Z_t(\mm^c)$ of
degree $\alpha$ such that
\begin{itemize}
\item[(iii)] $f=f_2 \mod Z_1(\mm^c)Z_{t-1}(\mm^c)$;
\item[(iv)] if a monomial $v [u_1,\dots,u_t]$ appears in $f_2$
and $X_n\mid u_1\cdots u_t$, then $X_n\mid v$.
\end{itemize}
Note that if $v [u_1,\dots, u_t]$ appears in $f_2$ and $X_n\nmid
u_1\cdots u_t$, then $X_n\mid v$ by degree reasons. Hence for every
monomial $v [u_1,\dots,u_t]$ appearing in $f_2$ we have $X_n\mid v$.
Therefore $f_2=X_ng$, and $g \in Z_t(\mm^c)$ has degree $<t(c+1)$.
This completes the proof.
\end{proof}

Next we describe some cycles which are needed in the following. For
$t \in \N$, $t\geq 1$ let $\mathcal{S}_t$ be the group of
permutations of $\{1,\dots,t\}$.

\begin{lem}
\label{lem:newcycles} Let $s, t$ be integers such that $1\leq s\leq
c$ and $t>0$. Let $a_1,a_2\dots,a_{t+1} \in S$ be monomials of
degree $s$ and $b_1,b_2\dots,b_{t}\in S$ monomials of degree $c-s$.
Then
\begin{equation}
\label{eq:cycle} \sum_{\sigma\in \mathcal{S}_{t+1}} (-1)^{\signum
\sigma} a_{\sigma(t+1)} [b_1a_{\sigma(1)}, b_2a_{\sigma(2)}, \dots,
b_ta_{\sigma(t)}]
\end{equation}
belongs to $Z_t(\mm^c)$.
\end{lem}
\begin{proof}
We apply the differential of $K(\mm^c)$ to (\ref{eq:cycle}) and
observe that for distinct integers $j_1,j_2,\dots,j_{i-1}, j_{i+1},
\dots, j_t$ in the range of $1$ to $t+1$ the free generator
$$
[b_1a_{j_1}, b_2a_{j_2}, \dots, b_{i-1}a_{j_{i-1}},
b_{i+1}a_{j_{i+1}} \dots, b_ta_{i_t}]
$$
appears twice in the image. The coefficients differ just by $-1$
because the corresponding permutations differ by a transposition.
Thus the element in (\ref{eq:cycle}) is indeed a cycle.
\end{proof}

\begin{rem}
\label{rem_specialcases} \
\begin{enumerate}
\item
Of course, it may happen that a cycle described in Lemma
\ref{lem:newcycles} is identically $0$. But for $t=1$ and $s=1$
these cycles  take the form
$$
X_i [bX_j] - X_j [bX_i],
$$
and, as said already in Remark \ref{genZ1}, they generate
$Z_1(\mm^c)$. For $s=c$ the cycles in Lemma \ref{lem:newcycles} are
the boundaries of $K_{t}(\mm^c)$ (multiplied by $t!$). Hence for
$c=1$ the cycles in \ref{lem:newcycles} generate the algebra
$Z(\mm)$. So there is some evidence that the cycles in Lemma
\ref{lem:newcycles} might generate $Z(\mm^c)$ in general.

\item
For $n=3, c=2, t=2$ and $s=1$ with $a_i=X_i$ for $i=1,2,3$ and
$b_i=X_i$ for $i=1,2$ the cycle in (\ref{eq:cycle}) is
$$
\begin{array}{cccccc}
  &         & &          & & 0 \\
 + & X_3[X_1^2 , X_2^2] & - & X_2[X_1^2, X_2X_3 ] & - &\overbrace{ X_3[X_1X_2, X_1X_2]} \\
 + & X_1[X_1X_2, X_2X_3] & + & X_2[X_1X_3, X_1X_2] & - & X_1[X_1X_3,
 X_2^2],
\end{array}
$$
a non-zero element in $Z_2(\mm^2)$.
\end{enumerate}
\end{rem}

Let $B_i(\mm^c)\subset Z_i(\mm^c)$ denote the $S$-module of
boundaries in $K_i(\mm^c)$.

\begin{thm}\label{thm:morevanishing}
We have
$$
(c+1)!\, \mm^{c-1}Z_1(\mm^c)^c \subset B_c(\mm^c).
$$
\end{thm}

\begin{proof}
For a monomial $b \in S$ of degree $c-1$ and variables $X_i,X_j$ we
set
$$
z_b(X_i,X_j)=X_i[b X_j]-X_j[b X_i].
$$
As observed in Remark \ref{genZ1}, the elements $z_b(X_i,X_j)$ generate
$Z_1(\mm^c)$. Let $a,b \in S$ be monomials of degree $c-1$. We note
that
$$
a z_{b}(X_i,X_j)+ b z_{a}(X_i,X_j) =\partial\left([aX_i,bX_j] +
[bX_i,aX_j]\right) \in B_1(\mm^c),
$$
that is,
\begin{equation}
\label{eq:van1} a z_{b}(X_i,X_j)=-b z_{a }(X_i,X_j) \mod B_1(\mm^c).
\end{equation}

Let $b_1,\dots, b_{c+1}\in S$ be monomials of degree $c-1$, and let
$X_{ij}\in \{X_1,\dots,X_n\}$ for $i=1,\dots,c$ and $j = 0,1$ be
variables. By construction, the elements
$$
f=b_{c+1} \prod_{i=1}^c z_{b_i}(X_{i0},X_{i1}) \in Z_c(\mm^c)
$$
generate $\mm^{c-1}Z_1(\mm^c)^c$. We have to show that
$(c+1)!f\in B_c(\mm^c)$.

Let $\sigma\in \mathcal{S}_{c+1}$ be an arbitrary permutation. From
Equation (\ref{eq:van1}) and from the inclusion
$B_1(\mm^c)Z_{c-1}(\mm^c)\subset B_c(\mm^c)$ it follows that
$$
f= (-1)^{\signum \sigma} b_{\sigma(c+1)} \prod_{i=1}^c
z_{b_{\sigma(i)}}(X_{i0},X_{i1}) \mod B_c(\mm^c).
$$
Hence
\begin{equation}
\label{eq:van2}
(c+1)! f =
\sum_{\sigma \in \mathcal{S}_{c+1}}
(-1)^{\signum \sigma} b_{\sigma(c+1)} \prod_{i=1}^c
z_{b_{\sigma(i)}}(X_{i0},X_{i1})
 \mod B_c(\mm^c).
\end{equation}
In the right-hand side of (\ref{eq:van2}) we replace
$z_{b_{\sigma(i)}}(X_{i0},X_{i1})$ with
$X_{i0}[b_{\sigma(i)}X_{i1}]-X_{i1}[b_{\sigma(i)}X_{i0}]$, then
expand the product and collect the multiples of $X_{1j_1}\cdots
X_{cj_c}$ for $j=(j_1,\dots, j_c)\in \{0,1\}^c$. We can rewrite
Equation (\ref{eq:van2}) as
\begin{equation}
\label{eq:van3} (c+1)! f = \sum_{j\in \{0,1\}^c}
(-1)^{j_1+\dots+j_c} X_{1j_1}\cdots X_{cj_c} W_j
 \mod B_c(\mm^c),
\end{equation}
where
$$
W_j= \sum_{\sigma \in \mathcal{S}_{c+1}} (-1)^{\signum \sigma}
b_{\sigma(c+1)} [ X_{1i_1} b_{\sigma(1)}, \dots, X_{ci_c}
b_{\sigma(c)}]
$$
with $i_k=1-j_k$ for $k=1,\dots,c$. Lemma \ref{lem:newcycles} yields
$W_j\in Z_{c}(\mm^c)$. Since $\mm^c Z_c (\mm^c) \subset B_c(\mm^c)$
we get
$$
X_{1j_1}\cdots X_{cj_c} W_j =0 \mod B_c(\mm^c).
$$
Thus Equation (\ref{eq:van3}) implies $(c+1)! f\in B_c(\mm^c)$ as
desired.
\end{proof}

As a consequence we obtain:

\begin{cor}\label{greengoes2}
The homology $H_t(\mm^c)_{tc+j}$ vanishes if $j\geq t+c$.
If  $t\geq c$ and the characteristic of $K$ is either $0$ or $>c+1$,
then $H_t(\mm^c)_{tc+j}=0$ for $j\geq t+c-1$.
\end{cor}

\begin{proof}
The first statement is a special case of Corollary \ref{greengoes}.
For the second, set $j=t+c-1$. We have to prove that
$H_t(\mm^c)_{tc+j}=0$. Theorem \ref{thm:generators} implies that
$Z_t(\mm^c)$ is generated by some elements $z_i$ of degree $<t(c+1)$
and by some elements $w_i$ of $Z_1(\mm^c)^t$ of degree $t(c+1)$.
Hence an element $f\in Z_t(\mm^c)_{tc+j}$ can be written as $f=\sum
a_iz_i+\sum b_iw_i$ with $a_i\in \mm^c$ and $b_i\in \mm^{c-1}$ by
degree reasons. Now $\sum a_iz_i\in \mm^c Z_t(\mm^c)\subset
B_t(\mm^c)$. In view of Theorem \ref{thm:morevanishing} we furthermore have
$$
 \sum b_iw_i\in \mm^{c-1}Z_1(\mm^c)^t=\mm^{c-1}Z_1(\mm^c)^cZ_1(\mm^c)^{t-c}
 \subset B_c(\mm^c)Z_1(\mm^c)^{t-c}\subset B_t(\mm^c).
$$
Summing up, $f\in B_t(\mm^c)$ and hence $H_t(\mm^c)_{tc+j}=0$.
\end{proof}

\begin{rem}\label{Weneedc+1}
The coefficient $(c+1)!$ in Theorem \ref{thm:morevanishing} and the
assumption on the characteristic in Corollary \ref{greengoes2} are necessary.
For $n=7, c=2$ and $\chara K=3$ we have checked that $\mm
Z_1(\mm^2)^2 \not\subset B_2(\mm^2)$ and that $\dim H_2(\mm^2)_7=1$.
More precisely, $H_2(\mm^2)$ has dimension $1$ in the multidegree
$(1,1,1,1,1,1,1)$ if $\chara K=3$.
\end{rem}

Another consequence of Theorem \ref{thm:morevanishing} is the
following:

\begin{cor}
Assume $\chara K$ is $0$ or $>c+1$. Then $\reg Z_{t+1}(\mm^c)\leq
(t+1)(c+1)-1$ for every $t\geq c$. In particular,
$Z_1(\mm^c)^{c+1}\subset \mm Z_{c+1}(\mm^c)$.
\end{cor}

\begin{proof} To prove the first assertion, let us denote by $Z_t$
the module $Z_t(\mm^c)$ and similarly for $B_t, H_t$ and $K_t$. The
short exact sequences
$$
0\to B_t \to Z_t \to H_t\to 0 \mbox{ and } 0\to Z_{t+1} \to K_{t+1}
\to B_t\to 0
$$
imply that $\reg(Z_{t+1})\leq \max\{ \reg(Z_t)+1, \reg(H_t)+2\}$.
Using Lemma \ref{regZ} and Corollary \ref{greengoes2} one obtains
$\reg(Z_{t+1})\leq (t+1)(c+1)-1$ for every $t\geq c$. The second
assertion follows immediately from the first.
\end{proof}

 \section{The Green-Lazarsfeld index of Veronese subrings of polynomial rings}
\label{sec:polynomial}

Again we consider a standard graded $K$-algebra $R$ of the form
$R=S/I$ where $K$ is a field, $S=K[X_1,\ldots,X_n]$ is a polynomial
ring over $K$ graded by $\deg(X_i)=1$ and $I\subset S$ is a graded
ideal.

 Given $c \in \N, c\geq 1 $ and $0\leq k<c$, we set
$$
V_R(c,k)=\bigoplus_{i \in \N} R_{k+ic}.
$$
Observe that $R^{(c)}=V_R(c,0)$ is the usual $c$-th {Veronese
subring} of $R$, and that the $V_R(c,k)$ are $R^{(c)}$-modules known
as the \emph{Veronese modules} of $R$. For a finitely generated
graded $R$-module $M$ we similarly define
$$
M^{(c)}=\bigoplus_{i \in \Z} M_{ic}.
$$
We consider $R^{(c)}$ as a standard graded $K$-algebra with
homogeneous component of degree $i$ equal to $R_{ic}$, and $M^{(c)}$
as a graded $R^{(c)}$-module with homogeneous components $M_{ic}$.
The grading of the $R^{(c)}$-module $V_R(c,k)$ is given by
$V_R(c,k)_i=R_{k+ic}$. In particular, the latter modules are all
generated in degree $0$ with respect to this grading.

Let $T=\Sym(R_c)$ be the symmetric algebra on vector space $R_c$,
that is,
$$
T=K[Y_u: u \in B_c]
$$
where $B_c$ is any $K$-basis of $R_c$. When $R=S$ the basis $B_c$
can be taken as the set of monomials of degree $c$. The canonical
map $T\to R^{(c)}$ is surjective, and the modules $V_R(c,k)$ are
also finitely generated graded $T$-modules (generated in degree
$0$).

With the notation of the preceding sections we have:

\begin{lem}\label{lem:Bettiinterpreter}
For $i\in \N$, $j\in \Z$ and $0\leq k<c$ we have
$$
\beta_{i,j}^T(V_R(c,k))=\dim_K H_{i}(\mm^c,R)_{jc+k}.
$$
\end{lem}

\begin{proof}
Let $K(T_1)$ be the Koszul complex (of $T$-modules) associated to
the elements $Y_u$ with $u\in B_c$. We observe that
$$
\beta_{i,j}^T(V_R(c,k))=\Tor_i^T(K,V_R(c,k))_{j} =\dim_K
H_i(K(T_1))\otimes_T V_R(c,k))_{j}.
$$
But the last homology is $H_{i}(\mm^c)_{jc+k}$, the $i$-th
homology of the complex $K(\mm^c)_{jc+k}$.
\end{proof}

Lemma \ref{lem:Bettiinterpreter} and Corollary \ref{greengoes2}
imply:

\begin{cor}\label{cor:ncpolynomial}
For all integers $i\geq 0$ and $k=0,\dots,c-1$ we have
$$
t_i^T(V_S(c,k))<1+i+\frac{i-k}{c}.
$$
If $K$ has characteristic $0$ or $>c+1$ and $i\geq c$, then
$$
t_i^T(V_S(c,k))<1+i+\frac{i-k-1}{c} .
$$
\end{cor}

\begin{rem}
\label{Andersen} Let $S=K[X_1,\dots,X_n]$. Andersen \cite{And}
proved that the graded Betti numbers $\beta_{ij}^T(S^{(2)})$ do not
depend on the characteristic of $K$ if $i\le4$ or if $i=5$ and
$n\leq 6$. She also proved that, for $n\geq 7$, one has $\beta_{5,
7}^T(S^{(2)})\neq 0$ in characteristic $5$ while $\beta_{5,
7}^T(S^{(2)})=0$ in characteristic $0$. Thus, for $n\geq 7$ one has
$$
\ind(S^{(2)})=\begin{cases}
                          5, & \chara K= 0,\\
                          4, & \chara K= 5.
              \end{cases}
$$
Also note that already $\beta_{2,3}(V_S(2,1))$ depends on the
characteristic if $n\geq 7$, as follows from the data in Remark
\ref{Weneedc+1}.
\end{rem}

We now record a duality on $H(\mm^c)$. It can be seen as an
Avramov-Golod type duality (see \cite[Theorem 3.4.5]{BRHE98}) or as
an algebraic version of Serre duality.

\begin{prop}\label{prop:duality}
Let $N=\binom{n+c-1}{c}.$ Then
$$
\dim_K H_i(\mm^c)_j=\dim_K H_{N-n-i}(\mm^c)_{Nc-n-j},\qquad
i,j\in\Z,\ i,j\ge 0.
$$
\end{prop}
\begin{proof}
For this proof (and only here) we consider the grading on the
polynomial ring $T=K[Y_u: u\in S \text{ monomial, } \deg u=c]$ in
which $Y_u$ has degree $c$. The polynomial ring $S$ in its standard
grading is a finitely generated graded $T$-module as usual.

Note that the canonical module of $S$ is $\omega_S=S(-n)$, and that
the canonical module of $T$ is $\omega_T=T(-Nc)$. Recall that
$$
\Ext_T^{j}(S,T(-Nc)) =
\begin{cases}
0&\text{if } j<N-n,\\
S(-n)&\text{if } j=N-n.
\end{cases}
$$
(See, e.g., \cite[Theorem 3.3.7 and Theorem 3.3.10]{BRHE98}.) Let $F$
be a minimal graded free $T$-resolution of $S$. Computing
$\Ext_T^{i}(S,T(-Nc))$ via $\Hom_T(F,T(-Nc))$, the minimal graded
free $T$-resolution of $S(-n)$, we see immediately that
$\beta_{i,j}^T(S)=\beta_{N-n-i,Nc-j}^T(S(-n))$. Then
\begin{equation*}
\dim_K H_i(\mm^c)_j
= \beta_{i,j}^T(S)\\
=\beta_{N-n-i,Nc-n-j}^T(S) =\dim_K
H_{N-n-i}(\mm^c)_{Nc-n-j}.\qedhere
\end{equation*}
\end{proof}

\begin{ex}
\label{ex:cn3} Let $\chara K=0$. Computer algebra systems as CoCoA
\cite{CoCoA}, Macaulay 2 \cite{MC2} or Singular \cite{Sing} can
easily compute the following diagram for $\dim_K H(\mm^3)$ in the
case $n=3$:
$$
\begin{tabular}{c|ccccccccc}
 & 0 & 1 & 2 & 3 & 4 & 5 & 6 & 7 \\
\hline
 0 & 1 & - & - & - & - & - & - & - & $\leftarrow$ \\
 1 & 3 & 15 & 21 & - & - & - & - & - \\
 2 & 6 & 49 & 105 & 147 & 105 & 21 & - & - \\
 3 &\bo& 27 & 105 & 189 & 189 & 105 & 27 & - & $\leftarrow$ \\
 4 & - &\bo & 21 & 105 & 147 & 105 & 49 & 6 \\
 5 & - & - &\bo & 0 & - & 21 & 15 & 3 \\
 6 & - & - & -  & \bo & 0 & - & - & 1 &$\leftarrow$ \\
\end{tabular}
$$
 The $(i,j)$-entry of the table is
$\dim_K H_i(\mm^c)_{ic+j}$ and - indicates that entry is $0$.
By selecting the rows whose indices are multiples of $c=3$ (those
marked by the arrows in the diagram) one gets the Betti diagram of
$S^{(3)}$. Green's theorem \cite[Thm.2.2]{GR2} implies the vanishing
in the positions of the boldface zeros and below. Our result implies
the vanishing in the positions of the non-bold zeros and below.
(Also see Weyman \cite[Example 7.2.7]{WEY} for the case $n=c=3$.)
\end{ex}

Using the duality we prove the upper bound for $\ind(S^{(c )})$ due
to Ottaviani and Paoletti \cite{OTPA} (in arbitrary characteristic).
To this end we need a variation of \cite[Corollary 2.10]{RO02}.
\begin{prop}\label{prop:cyclecounter}
Let $(e_i:i=1,\dots, m)$ be a basis of the vector space $\bigwedge^t
S_c$ (thus $m=\binom{N}{t}$ with $N=\binom{n-1+c}{n-1}$). Let
$$
z=\sum_{i=1}^m f_i e_i
$$
be a non-zero element in $Z_t(\mm^c)$. Then the $K$-vector space
generated by the coefficients $f_i$ of $z$ has dimension $\geq t+1$.
\end{prop}

\begin{proof}
Since the $K$-vector space dimension of the space of coefficients does not depend
on the basis, it is enough to prove the assertion for the monomial basis
$(e_i)$. We use induction on $t$.

The case $t=0$ is obvious. So assume $t>0$. Fix a term order, for
example the lexicographic term order, on $S$. Let $C(z)$ denote the
vector space generated by the coefficients of $z$. As already
discussed in the proof of Theorem \ref{thm:generators}, for every
monomial $u$ of degree $c$ we may write $z=a+b[u]$ with $b\in
Z_{t-1}(\mm^c)$. Choose $u$ to be the largest monomial (with respect
to the term order) such that the corresponding $b$ is non-zero. By
induction $\dim_K C(b)\geq t$ and $C(b)\subset C(z)$.

If $C(b)\neq C(z)$ then clearly $\dim_K C(z)\geq t+1$. If instead
$C(b)=C(z)$, then $C(a)\subset C(b)$. Let $v$ be the largest
monomial appearing in the elements of $C(b)$. The inclusion
$C(a)\subset C(b)$ implies that every monomial appearing in the
elements of $C(a)$ is $\leq v$. But $\partial(a)\pm bu=0$ and hence
$C(\partial(a))=C(bu)=uC(b)$. The monomial $vu$ appears in $C(bu)$.
Every monomial in $C(\partial(a))$ is of the form $wu_1$ where $w$
is a monomial appearing in $C(a)$ and $u_1$ is a monomial of degree
$c$ which is an ``exterior'' factor of some free generator appearing
in $z$. By construction $w\leq v$ and $u_1<u$. It follows that
$wu_1\neq vu$, a contradiction with $C(\partial(a))=uC(b)$.
\end{proof}

\begin{thm}
\label{thm:ottpa} For $n\geq 3$ and $c\geq 3$ one has
$\ind(S^{(c)})\leq 3c-3$, and equality holds for $n=3$.
\end{thm}

\begin{proof}
We first consider the case $n=3$. By an inspection of the Hilbert
function of (the Cohen-Macaulay ring) $S^{(c)}$ one sees immediately
that $\reg S^{(c)}\leq 2$, that is, $t_i^T(S^{(c)}) \leq i+2$ for
every $i\geq 0$. From Theorem \ref{lem:Bettiinterpreter} and
Proposition \ref{prop:duality} we have
$$
\beta_{i,j}^T(S^{(c)}) = \dim_K H_i(\mm^c)_{jc} = \dim_K
H_{N-3-i}(\mm^c)_{Nc-3-jc}.
$$
Therefore $t_i^T(S^{(c)}) \leq i+1$ if and only if
$$
H_{N-3-i}(\mm^c)_{(N-3-i)c+c-3}=0,
$$
and, since the boundaries have coefficients of degree $\geq c$, this
is equivalent to $$ Z_{N-3-i}(\mm^c)_{(N-3-i)c+c-3}=0.$$ So we have
to analyze the cycles in $Z_{N-3-i}(\mm^c)$ with coefficients of
degree $c-3$.

It follows from Proposition \ref{prop:cyclecounter} that
$$
N-3-i+1\leq \dim_K S_{c-3}= \binom{c-1}{2}
$$
if there exists a non-zero cycle $z\in Z_{N-3-i}(\mm^c)$ with
coefficients of degree $c-3$. Thus there are no cycles in that
degree if $N-3-i\geq \binom{c-1}{2}$. Hence
$$
t_i^T(S^{(c)}) \leq i+1 \text{ for } 0\leq i\leq 3c-3,
$$
that is, $\ind(S^{(c)})\geq 3c-3$. It remains to show that $S^{(c)}$
does not satisfy the property $N_{3c-2}$. We have to find a non-zero
cycle in $Z_{j}(\mm^c)$ with coefficients of degree $c-3$ where
$j=N-3-i$. Note that $j=\binom{c-1}{2}-1$ and so $j+1=\dim S_{c-3}$.
Take the monomials $u'_1,\dots,u'_{j+1}$ of degree $c-3$ and set
$u_k=u'_kX_1X_2X_3$ for $k=1,\dots,j+1$. Then
$$
w=\partial([u_1,\dots,u_{j+1}]) \in Z_j(\mm^c)
$$
is non-zero boundary with coefficients of degree $c$. But we can
divide each coefficient of $w$ by $X_1X_2X_3$ to obtain a non-zero
cycle $z\in Z_j(\mm^c)$ with coefficients of degree $c-3$. It
follows that $S^{(c)}$ does not satisfy the property $N_{3c-2}$.
This concludes the proof for $n=3$.

Now let $n>3$. Recall that $H_i(\mm^c)$ is multigraded. For a vector
$a=(a_1,\dots,a_n)\in \N^n$ with $a_i=0$ for $i>3$ let
$b=(a_1,a_2,a_3)$. We may identify
$$
 H_i(\mm^c)_a= H_i(\mm_3^c)_b
$$
where $H_i(\mm_3^c)$ is the corresponding Koszul homology in $3$
variables. Since for $n=3$ the $c$-th Veronese does not satisfy
$N_{3c-2}$ it follows that the same is true for all $n\geq 3$,
proving that $\ind(S^{(c)})\leq 3c-3$. \end{proof}

\begin{rem}
It is well-known that $\reg S^{(c)}\leq n-1$ in general, i.e.,
$t_i(S^{(c)})\leq i+n-1$. Analogously to the proof of Theorem
\ref{thm:ottpa} one can determine the largest $i$ such that
$t_i(S^{(c)})<i+n-1$. Again this is determined by elements in
$Z_i(\mm^c)$ with coefficients of degree $c-n$. It remains to count
the monomials of $S$ in that degree. For example, for $c\geq n=4$
one obtains $t_i(S^{(c)})<i+3$ if and only if $i\leq 2c^2-2$.
\end{rem}

\section{The Green-Lazarsfeld index of Veronese subrings of standard graded rings}
\label{sec:standard}

Let $D$ be a Koszul $K$-algebra and $I$ be a homogeneous ideal of
$D$. Set $R=D/I$. We want to relate the Green-Lazarsfeld index of
$R^{(c)}$ to that of $D^{(c)}$. For a polynomial ring $S$ Aramova,
B\v{a}rc\v{a}nescu and Herzog proved in \cite[Theorem 2.1]{ABH} that the Veronese modules
$V_S(c,k)$ have a linear resolution over the Veronese ring $S^{(c)}$. We show that this property holds for Koszul algebras in
general.

\begin{lem}\label{lem:t_iveronese}
Assume $D$ is a Koszul algebra and, for a given $c$, let
$T=\Sym(D_c)$ be the symmetric algebra of $D_c$.
\begin{itemize}
\item [(a)]
The Veronese module $V_D(c,k)$ has a linear resolution as a
$D^{(c)}$-module.
\item [(b)]
For every $k=0,\dots,c-1$ we have
$$
t^T_i(V_D(c,k)) \leq t^T_i(D^{(c)}).
$$
\end{itemize}
\end{lem}
\begin{proof}
(a) Let $\mm$ denote the homogeneous maximal ideal of $D$, and set
$A=D^{(c)}$ and $V_k=V_D(c,k)$. We prove by induction on $i$ that
$t_i^A(V_k)\leq i$ for all $i$ and $k$. For $i=0$ the assertion is
obvious and it is so for $k=0$ and $i\geq 0$, too. Assume that
$i>0$. The ideal $\mm^k$ is generated in degree $k$ and, since $D$
is Koszul, it has a linear resolution over $D$. Shifting that
resolution by $k$, we obtain a complex
$$
\cdots\to F_i\to F_{i-1}\to\dots \to F_1\to F_0\to 0
$$
resolving $\mm^k(k)$ and such that $F_i=D(-i)^{\beta_i}$. Applying
the exact functor $(\ )^{(c)}$ to it we get an exact complex of
$A$-modules
$$
\cdots\to F_i^{(c)} \to F_{i-1}^{(c)}\to\cdots \to F_1^{(c)}\to A^{\beta_0}\to V_k\to 0.
$$
Note that $D(-j)^{(c)}=V_e(-\lceil j/c\rceil)$ where $e=c \lceil
j/c\rceil -j$. Therefore $F_j^{(c)}=V_{e_j}(-\lceil
j/c\rceil)^{\beta_j}$ where $e_j=c \lceil j/c\rceil -j$. Applying Lemma
\ref{lem:tiineq} (b) to the complex above we have
$$
t_i^A(V_k)\leq \max\{ t_{i-j}^A(V_{e_j})+\lceil j/c\rceil : j=0,\dots,i\}.
$$
Obviously $ t_{i}^A(V_{e_0})=t_i^A(A)=-\infty$ and, by induction,
$t_{i-j}^A(V_{e_j})\leq i-j$ for $j=1,\dots, i$. Therefore
$$
t_i^A(V_k)\leq \max\{ i-j+\lceil j/c\rceil : j=1,\dots, i\}= i
$$
and this concludes the proof of (a). For (b) we may apply
Lemma \ref{lem:tiineq}(b) to the minimal $A$-free resolution of $V_k$
and to get the desired inequality.
\end{proof}

Now we prove the main result of this section.
\begin{thm}
\label{thm:rateresult} Assume $D$ is a Koszul algebra and $R=D/I$.
 Let $c\geq \slope _D(R)$. Then $\ind(R^{(c)})\geq \ind(D^{(c)})$.
\end{thm}

\begin{proof} To prove the statement we set $A=D^{(c)}$ and $B=R^{(c)}$.
By virtue of Lemma \ref{lem:tiineq} (d) it is enough to show that
$\reg_{A}(B)=0$. Let
$$
\dots \to F_p\to \dots \to F_1 \to F_0\to R\to 0
$$
be the minimal graded free resolution of $R$ over $D$.
 Since $(\ )^{(c)}$ is an exact functor, we obtain
an exact complex of finitely generated graded $A$-modules
\begin{equation}
\label{res-c} \dots \to F_p^{(c)} \to \dots \to F_1^{(c)} \to
F_0^{(c)} \to B \to 0.
\end{equation}
Hence by virtue of Lemma \ref{lem:tiineq} (b) we have
$$
\reg_{A}(B ) \leq \max \{ \reg_{A}(F_i^{(c)})-i : i \geq 0\}.
$$
Note that $D(-k)^{(c)}=V_D(c,e)(-\lceil k/c\rceil)$ where $e=c
\lceil k/c\rceil -k$. Hence, by virtue of Lemma
\ref{lem:t_iveronese}, $\reg_{A}(D(-k)^{(c)})=\lceil k/c\rceil $.
Therefore, since $F_i=\bigoplus_{k \in \Z} S(-k)^{\beta_{ik}^D(R)}$
we get $\reg_{A}(F_i^{(c)})=\lceil t_i^D(R) /c\rceil$. Summing up,
$$
\reg_{A}(B ) \leq \max\{ \lceil t_i^D(R) /c\rceil -i : i\geq 0\}.
$$
If $c\geq \slope _D(R)$, then $t_i^D(R)\leq ci$ and hence
$\reg_{A}(B)=0$. This concludes the proof.
\end{proof}

As a corollary we have:

\begin{cor}
Let $S$ be a polynomial ring and $R=S/I$ a standard graded algebra
quotient of it and let $c\geq \slope _S(R)$. Then $\ind(R^{(c)})\geq
\ind(S^{(c)})$. In particular,
\begin{itemize}
\item[(a)] $\ind(R^{(c)})\geq c$. Furthermore, if $K$ has
characteristic $0$ or $>c+1$, then we have $\ind(R^{(c)})\geq
c+1$.
\item[(b)] If $\dim R_1=3$, then $\ind(R^{(c)})\geq 3c-3$.
\end{itemize}
\end{cor}

Note that $\slope _S(R)=2$ if $R$ is Koszul; see \cite{ACI}.
Furthermore $\slope _S(R)\leq a$ if $R$ is defined by either a
complete intersection of elements of degree $\leq a$ or by a
Gr\"obner basis of elements of degree $\leq a$.

\begin{rem}
Sometimes the bound in Theorem \ref{thm:rateresult} can be improved
by a more careful argumentation. Let $R=S/I$ and let $T$ be the
symmetric algebra of $S_c$. For instance, using the argument of the
proof of Theorem \ref{thm:rateresult} one shows that
$$
t^T_i(R^{(c)})\leq \max\{ t^T_{i-j}(S^{(c)})+\lceil
t^S_j(R)/c\rceil\} : j=0,1,\dots,i\}.
$$
It follows that $\ind(R^{(c)})\geq p$ if $c\geq p$ and $\ind(R)\geq
p$, a result proved by Rubei in \cite{RU00}. It is very easy to show
that $R^{(c)}$ is defined by quadrics, i.e. $\ind(R^{(c)})\geq 1$
provided $c\ge t_1^S(R)/2$. Similarly, one can prove that
$\ind(R^{(c)})\geq p$ if
$$
c\geq \max\{ p, \ \ \max\{ t_{j}^S(R)/j : j=1,\dots, p-1\}, \ \
t_{p}^S(R)/(p+1)\}.
$$
\end{rem}

\begin{rem}\label{3rema}
\
\begin{itemize}
\item[(a)] Let us say that a positively graded $K$-algebra is \emph{almost
standard} if $R$ is Noetherian and a finitely generated module over
$K[R_1]$. If $K$ is infinite, then this property is equivalent to
the existence of a Noether normalization generated by elements of
degree $1$. Galliego and Purnaprajna \cite[Theorem 1.3]{GAPU99}
proved a general result on the property $N_p$ of Veronese
subalgebras of almost standard $K$-algebras $R$ of $\depth\ge 2$
over fields of characteristic $0$: $R^{(c)}$ has $N_p$ for all $c\ge
\max\{\reg(R)+p-1,\reg(R), p\}$. If $\reg(R)\ge 1$ and $p\ge 1$,
this amounts to the property $N_{c-\reg(R) +1}$ of $R^{(c)}$ for all
$c\ge\reg(R)$. Thus Theorem \ref{thm:rateresult} gives a stronger
result for standard graded algebras.

\item[(b)] Eisenbud, Reeves and Totaro \cite{EIRETO} proved that the
Veronese subalgebras $R^{(c)}$ of standard graded $K$-algebra $R$
are defined by an ideals with Gr\"obner bases of degree $2$ for all
$c\ge (\reg(R)+1)/2$. It follows that these algebras are Koszul.

\item[(c)] If $R$ is almost standard and Cohen-Macaulay, then $R^{(c)}$ is
defined by an ideal with a Gr\"obner bases of degree $2$ for every
$c\ge \reg(R)$. See Bruns, Gubeladze and Trung \cite[Theorem
1.4.1]{BGT} or Bruns and Gubeladze \cite[Theorem 7.41]{PRK}.
\end{itemize}
\end{rem}

\section{The multigraded case}
\label{sec:multigraded}

The results  presented in this paper have natural extensions to the
multigraded case. Here we just formulate the main statements.
Detailed proofs will be given in the forthcoming article \cite{BCR}.
Suppose $S=K[X^{(1)},\dots,X^{(m)}]$ is a $\Z^m$-graded polynomial
ring in which each $X^{(i)}$ is the set of variables of degree
$e_i\in \Z^m$. For a vector $c\in (c_1,\dots,c_m)\in \N_+^m$
consider the $c$-th diagonal subring $S^{(c)}=\dirsum_{i\in \N}
S_{ic}$, the coordinate ring of the corresponding Segre-Veronese
embedding. The following result improves the bound of Hering,
Schenck and Smith \cite{HSS} by one:

\begin{thm}
With the notation above one has: $\min(c)\leq \ind(S^{(c)})$.
Moreover, we have $\min(c)+1\leq \ind(S^{(c)})$ if $\chara K=0$ or $\chara
K>1+\min(c)$.
\end{thm}

Similarly one has the multigraded analog of Theorem \ref{thm:rateresult}.
Here one uses the fact, proved in \cite{CHTV}, given any
$\Z^m$-graded standard graded algebra quotient of $S$ then if the $c_j$'s
are big enough (in terms of the multigraded Betti numbers of $R$
over $S$) then $\reg_{S^{(c)}}(R^{(c)})=0$.

\begin{prop}
Assume that for all $j=1\dots,m$ one has $c_j\geq \max\{ \alpha_j/i:
i>0, \alpha\in \Z^m \text{ and } \beta_{i,\alpha}^S(R)\neq 0,\} $
then $\ind(R^{(c)})\geq \ind(S^{(c)})$.
\end{prop}

\end{document}